\documentclass[reqno, a4paper, 11pt]{amsart}
\usepackage{amsmath,amsthm,amssymb}
\usepackage{graphicx}
\usepackage{time} 
\usepackage{comment}
\usepackage{thmtools}
\usepackage{thm-restate}

\usepackage[dvipdfmx]{hyperref}
\hypersetup{
	setpagesize=false,
	bookmarksnumbered=true,%
	bookmarksopen=true,%
	colorlinks=true,%
	linkcolor=magenta,
	citecolor=blue,
}
\theoremstyle{plain}
\newtheorem{theorem}{Theorem}[section]
\newtheorem{lemma}[theorem]{Lemma}
\newtheorem{corollary}[theorem]{Corollary}
\newtheorem{proposition}[theorem]{Proposition}
\theoremstyle{definition}
\newtheorem{definition}[theorem]{Definition}
\newtheorem{remark}[theorem]{Remark}
\newcommand{\ZZ}{\mathbb{Z}}
\newcommand{\CC}{\mathbb{C}}
\newcommand{\calT}{\mathcal{T}}
\newcommand{\calD}{\mathcal{D}}
\newcommand{\Int}{\mathrm{int\,}}

\title[Trisection genus of knot traces]{Trisection genus of knot traces}
\author[Natsuya Takahashi]{Natsuya Takahashi}
\date{May 28, 2026. \textit{Revised}: September 17, 2026.}
\subjclass[2020]{57R65}
\keywords{4-manifold, knot trace, trisection}
\address{Department of Mathematics, College of Science and Technology, Nihon University, 1-8-14, Kanda-Surugadai, Chiyoda-ku, Tokyo 101-8308, Japan}
\email{takahashi.natsuya@nihon-u.ac.jp}
\begin{document}
\begin{abstract}
We classify knot traces with trisection genus at most 2.
We give infinitely many torus knots whose traces have trisection genus 3, and infinitely many hyperbolic knots whose traces have trisection genus 4.
We also show that there exist infinite families of knots whose traces have arbitrarily large trisection genus.
In addition, we determine or give sharp bounds for the trisection genus of the traces of several well-known knots, such as the figure-eight knot, the $(p, pq+1)$-torus knots, and the $(-2, 3, 2n+1)$-pretzel knots.
\end{abstract}
\maketitle

\section{Introduction}

It is a fundamental problem in low-dimensional topology to classify $n$-dimensional manifolds by their natural invariants.
For smooth $4$-manifolds, the trisection genus is a natural integer-valued invariant arising from the notion of a trisection, introduced by Gay and Kirby~\cite{GayKir16}.
The trisection genus $g(X)$ of a $4$-manifold $X$ is defined as the minimum genus of the triple intersection surface among all trisections of $X$; thus, it can be viewed as a $4$-dimensional analogue of the Heegaard genus of $3$-manifolds.


In the early stages of trisection theory, it was shown in~\cite{GayKir16} and \cite{MeiZup17_1} that closed $4$-manifolds with trisection genus at most $2$ are limited to finitely many standard ones.
On the other hand, there exist infinitely many closed $4$-manifolds with trisection genus $3$.
Although genus-$3$ manifolds have been studied in \cite{Mei18} and \cite{AraZup25a}, a complete classification has not yet been achieved.
Determining the trisection genus is also challenging for $4$-manifolds with boundary.
While it is easy to see that $4$-manifolds admitting a genus-$0$ relative trisection are exactly $4$-dimensional $1$-handlebodies, little is known about the higher-genus case.


We study the problem of classifying 4-manifolds with boundary by their trisection genus.
A natural approach to this problem is to restrict the class of $4$-manifolds under consideration.
In this paper, we focus on knot traces, a well-studied class of simply-connected $4$-manifolds with boundary.
For a knot $K$ in the $3$-sphere $S^3$ and an integer $m$, the $m$-trace $X_m(K)$ of $K$ is the $4$-manifold obtained by attaching a single $4$-dimensional $2$-handle to the $4$-ball $D^4$ along $K\subset\partial{D^4}=S^3$ with framing $m$.
The boundary $\partial X_m(K)$ is diffeomorphic to the closed $3$-manifold obtained by Dehn surgery on $S^3$ along the knot $K$ with coefficient $m$. We denote this $3$-manifold by $S^3_m(K)$.

%
Our first main result is a complete classification of knot traces $X_m(K)$ satisfying $g(X_m(K)) \leq 2$.

\begin{restatable}{theorem}{maingonetwo}\label{thm:g=12}
Let $K$ be a knot in $S^3$ and let $m$ be an integer.
\begin{itemize}
\item There is no knot $K$ and no integer $m$ such that $g(X_m(K))=0$.
\item $g(X_m(K))=1$ if and only if $K$ is the unknot $U$ and $m=\pm 1$.
\item $g(X_m(K)) = 2$ if and only if either $K$ is the unknot $U$ with $m \neq \pm 1$, the right-handed trefoil knot with $m = 5$, or the left-handed trefoil knot with $m = -5$.
\end{itemize}
\end{restatable}

In contrast to the cases above, there exist infinitely many knots whose traces have trisection genus $3$.

\begin{restatable}{theorem}{maingthree}\label{thm:g=3}
There exist infinitely many torus knots $K$ such that $g(X_m(K))=3$ for infinitely many integers $m$.
Furthermore, for any knot $K$ and integer $m$, if $g(X_m(K))=3$, then the boundary $S^3_m(K)$ is homeomorphic to the closed $3$-manifold $M(0;a_1,a_2,a_3)$ for some $a_1, a_2, a_3 \in \ZZ$.
\end{restatable}

We denote by $M\left(e;{a_1}/{b_1},{a_2}/{b_2},{a_3}/{b_3}\right)$ the $3$-manifold given by the surgery diagram shown in Figure~\ref{fig:Kd-Seif3sing}, where $e, a_i,b_i\in\mathbb{Z}$ and $\gcd(a_i,b_i)=1$.
Note that $M(0;a_1,a_2,a_3)$ is the connected sum of two lens spaces of the form $L(*,1)$ if at least one of $a_1,a_2,a_3$ is $0$.
%
%
Conversely, if $a_1 a_2 a_3\neq0$, then $M(0;a_1,a_2,a_3)$ is irreducible and, in particular, is Seifert fibered (see e.g. \cite{Ari08}).
\begin{figure}[!htbp]
\centering
\includegraphics[scale=1]{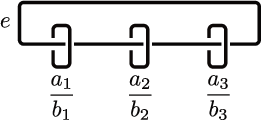}
\caption{A surgery diagram of the closed $3$-manifold $M\left(e;\frac{a_1}{b_1},\frac{a_2}{b_2},\frac{a_3}{b_3}\right)$.}
\label{fig:Kd-Seif3sing}
\end{figure}

We give specific examples of knot traces that satisfy the properties of Theorems ~\ref{thm:g=12} and \ref{thm:g=3}.
The next proposition determines the trisection genus of traces of the unknot.

\begin{restatable}{proposition}{mainunknot}\label{thm:tg-unknot}
The trisection genus of the $m$-trace of the unknot $U$ is given by
\begin{equation*}
g(X_m(U)) =
	\begin{cases}
		1 & \text{if $m = \pm 1$}, \\
		2 & \text{otherwise}.
	\end{cases}
\end{equation*}
\end{restatable}

The upper bounds for $g(X_m(U))$ were previously given by Castro, Gay, and Pinz\'{o}n-Caicedo~\cite{CasGayPin18_1}.
They constructed an explicit genus-2 relative trisection of $X_m(U)$.
We complete the proof of Proposition~\ref{thm:tg-unknot} by giving the corresponding lower bounds.

The following theorem provides an infinite family of torus knots whose traces have trisection genus $3$, thereby realizing the first part of Theorem~\ref{thm:g=3}.

\begin{restatable}{theorem}{maintorus}\label{thm:tg-Tp,p+1}
Let $p\geq2$ and $m$ be integers. The trisection genus of the $m$-trace of the $(p,p+1)$-torus knot $T_{p,p+1}$ is given by
\begin{equation*}
g(X_m(T_{p,p+1}))=
	\begin{cases}
		2 & \text{if } m=5 \text{ and } p=2, \\
		3 & \text{otherwise.}
	\end{cases}
\end{equation*}
For the mirror image $-T_{p,p+1}$, we have
\begin{equation*}
g(X_m(-T_{p,p+1}))=
	\begin{cases}
		2 & \text{if } m=-5 \text{ and } p=2, \\
		3 & \text{otherwise.}
	\end{cases}
\end{equation*}
\end{restatable}

We construct an explicit relative trisection diagram realizing the minimal genus of $X_m(T_{p,p+1})$ for any $m\in\ZZ$ and $p\in\ZZ_{\geq2}$.
%
%
Theorem~\ref{thm:tg-Tp,p+1} implies the following corollary.

\begin{corollary}
There exist infinitely many knots $K$ such that the trisection genus $g(X_m(K))$ is constant for all integers $m$.
\end{corollary}

We also consider more general torus knots of the form $T_{p,pq+1}$ and give an upper bound for the trisection genus of their traces.

\begin{restatable}{theorem}{maintoruspq}\label{thm:tg-Tp,pq+1}
For any integers $p, q,$ and $m$, the trisection genus of the $m$-trace of the $(p, pq+1)$-torus knot $T_{p, pq+1}$ is at most $4$.
\end{restatable}

However, we have not been able to find any of these torus knots $T_{p,pq+1}$ whose traces have trisection genus exactly $4$.
The next theorem shows the existence of knot traces with trisection genus 4.

\begin{restatable}{theorem}{maingfour}\label{thm:g=4}
There exist infinitely many hyperbolic knots $K$ such that $g(X_m(K))=4$ for infinitely many integers $m$.
Furthermore, for any hyperbolic knot $K$, the inequality $g(X_m(K)) \geq 4$ holds if the boundary $S^3_m(K)$ is not homeomorphic to $M(0;a_1,a_2,a_3)$ for any $a_1, a_2, a_3 \in \ZZ$.
In particular, this inequality holds for all but finitely many integers $m$.
%
\end{restatable}

The final statement follows from the fact that $M(0;a_1,a_2,a_3)$ is not hyperbolic, together with Thurston's hyperbolic surgery theorem~\cite{Thu22}.

As examples satisfying the first statement of Theorem~\ref{thm:g=4}, we consider the $(-2,3,2n+1)$-pretzel knots (see Figure~\ref{fig:K_n}), denoted by $P(-2,3,2n+1)$.
It is known that $P(-2,3,2n+1)$ is hyperbolic if $n\notin\{0,1,2\}$, while the remaining cases are given by $P(-2,3,1)=T_{2,5}$, $P(-2,3,3)=T_{3,4}$, and $P(-2,3,5)=T_{3,5}$ (see e.g. \cite{Oer84}, \cite{BonSie10}).

\begin{restatable}{theorem}{mainpretzel}\label{thm:pretzel}
For any integers $m$ and $n$ with $n\notin\{0,1,2\}$, the trisection genus of $X_m(P(-2,3,2n+1))$ is $4$.
\end{restatable}

The upper bound is obtained by constructing an explicit genus-$4$ relative trisection diagram for $X_m(P(-2,3,2n+1))$.
For the lower bound, we show that the boundary $S^3_m(P(-2,3,2n+1))$ is not homeomorphic to $M(0;a_1,a_2,a_3)$ for any $a_1,a_2,a_3\in\mathbb{Z}$.




Regarding the figure-eight knot $4_1$, we obtain the next result, which also gives examples of knot traces with trisection genus $4$.

\begin{restatable}{theorem}{mainfigureeight}\label{thm:tg-fig8}
If $m\notin\{\pm1,\pm2 \}$, then $g(X_m(4_1))=4$. If $m\in\{\pm1,\pm2 \}$, then $g(X_m(4_1))$ is either $3$ or $4$.
\end{restatable}

The boundary $S^3_{m}(4_1)$ is a Seifert fibered space of the form $M(0;a_1,a_2,a_3)$ if $m \in \{\pm 1, \pm 2 \}$.
In these cases, $X_m(4_1)$ could potentially admit a genus-$3$ relative trisection (see Remark~\ref{rem:4_1Seifert}).
However, we have not been able to show the existence or non-existence of such a trisection.

Finally, we show that the trisection genus of a knot trace can be arbitrarily large.

\begin{restatable}{theorem}{mainlargen}\label{thm:g>N}
For any positive integer $N$, there exists a knot $K$ such that $g(X_{m}(K)) > N$ for infinitely many integers $m$.
\end{restatable}

As an example of a knot that satisfies this theorem, we can take the $(3,3,\ldots,3)$-pretzel knot with an odd number of strands at least $N$ (see Figure~\ref{fig:pretzel333}).

To determine the trisection genus of knot traces, we use two complementary approaches.
One is to obtain lower bounds of the trisection genus.
For this purpose, in Section~\ref{sect:lower}, we study lower bounds for the trisection genus arising from properties of the boundary $3$-manifolds.
The other approach is to give upper bounds by explicitly constructing low genus relative trisections of knot traces. This is carried out in Section~\ref{sect:determining}.
Finally, in Section~\ref{sect:proofmain}, we complete the proofs of the main results.

\subsection*{Notation and conventions}

Throughout this paper, all manifolds and surfaces are assumed to be compact, connected, oriented, and smooth.
If two smooth manifolds $X$ and $Y$ are orientation-preserving diffeomorphic, then we write $X \cong Y$.
We denote the Euler characteristic of a manifold $X$ by $\chi(X)$.
Let $\Sigma_{g,b}$ denote a surface of genus $g$ with $b$ boundary components.
We define the lens space $L(p, q)$ as the $3$-manifold obtained by $-p/q$-surgery on the unknot $U$, that is, $L(p, q):=S^3_{-p/q}(U)$.

\section{Relative trisections}

In this section, we give a brief overview of some properties of relative trisections. 
A relative trisection was introduced by Gay and Kirby~\cite{GayKir16} for $4$-manifolds with boundary.
For a precise definition, we refer the reader to \cite{CasGayPin18_1}.
Here, we briefly recall several key properties.

Let $g, k, p,$ and $b$ be integers satisfying the inequalities $g, k, p \geq 0$, $b \geq 1$, and $2p+b-1 \leq k \leq g+p+b-1$.
For a compact, connected, oriented, smooth $4$-manifold $X$ with non-empty connected boundary, if a decomposition $X = X_1 \cup X_2 \cup X_3$ is a $(g, k; p, b)$-relative trisection, then the following properties hold:
\begin{itemize}
%
%
\item
For each $i\in\{1,2,3\}$, the sector $X_i$ is diffeomorphic to the genus-$k$ $4$-dimensional $1$-handlebody $\natural^{k}(S^1\times D^3)$.
\item
For each $i,j\in\{1,2,3\}$ such that $i\neq j$, the double intersection $X_i\cap X_j(=\partial{X_i}\cap\partial{X_j})$ is diffeomorphic to a compression body from $\Sigma_{g,b}$ to $\Sigma_{p,b}$.
This compression body is also diffeomorphic to the $3$-dimensional $1$-handlebody $\natural^{g+p+b-1}(S^1\times D^2)$.
\item
The triple intersection $X_1\cap X_2\cap X_3$ is diffeomorphic to the genus-$g$ surface $\Sigma_{g,b}$ with $b$ boundary components.
\end{itemize}

A trisection gives rise to natural invariants.
The \textit{genus} of a (relative) trisection $X_1 \cup X_2 \cup X_3$ is defined as the genus of the triple intersection surface $X_1 \cap X_2 \cap X_3$.
For a $4$-manifold $X$, the \textit{trisection genus} $g(X)$ is defined as the minimum genus of the triple intersection surface among all (relative) trisections of $X$.
It is known that the trisection genus is a diffeomorphism invariant of smooth $4$-manifolds.

We collect here several properties of relative trisections used in this paper.

\begin{lemma}[{\cite[Lemma~17]{GayKir16}}, see also {\cite[Lemma~11]{CasGayPin18_1}}]\label{lem:obd}
A $(g,k;p,b)$-relative trisection $X=X_1\cup X_2\cup X_3$ of a $4$-manifold $X$ induces an open book decomposition on $\partial{X}$ with page $\Sigma_{p,b}$.
%
\end{lemma}

\begin{theorem}[{\cite[Theorem~20]{GayKir16}}]
For any $4$-manifold $X$ with connected boundary and any open book decomposition of $\partial{X}$, there exists a relative trisection of $X$ that induces the given open book decomposition.
\end{theorem}

\begin{proposition}[{\cite[Corollary~2.10]{CasOzb19}}]\label{prop:Euler-rel}
If a $4$-manifold $X$ with boundary admits a $(g,k;p,b)$-relative trisection, then its Euler characteristic is given by $\chi(X)=g-3k+3p+2b-1$.
\end{proposition}

A relative trisection can be described by a relative trisection diagram, which consists of a $4$-tuple $(\Sigma;\alpha, \beta, \gamma)$ of a surface and three sets of curves.
We recall the definition of a relative trisection diagram as given by Castro, Gay, and Pinz\'{o}n-Caiced~\cite{CasGayPin18_1}.

\begin{definition}
Let $\Sigma$ and $\Sigma'$ be compact, connected, oriented surfaces with boundary.
For $i\in \{1,\ldots,n\}$, let $\eta^i=\{\eta^i_1,\ldots,\eta^i_k\}$ and $\zeta^i=\{\zeta^i_1,\ldots,\zeta^i_k\}$ be sets of $k$ pairwise disjoint simple closed curves on $\Sigma$ and $\Sigma'$, respectively.
Then, we consider two $(n+1)$-tuples $\mathcal{D}:=(\Sigma;\eta^1,\ldots,\eta^n)$ and $\mathcal{D}':=(\Sigma';\zeta^1,\ldots,\zeta^n)$.
\begin{itemize}
\item
Suppose that $\Sigma=\Sigma'$. The two $(n+1)$-tuples $\mathcal{D}$ and $\mathcal{D}'$ are called \textit{isotopic} if for any $i\in\{1,\ldots,n\}$ there exists an ambient isotopy $\{\varphi_t^i:\Sigma\to\Sigma\}_{t\in[0,1]}$ such that $\varphi_1^i(\eta^i)=\zeta^i$.
\item
Suppose that $\Sigma=\Sigma'$. The two $(n+1)$-tuples $\mathcal{D}$ and $\mathcal{D}'$ are called \textit{slide-equivalent} if for each $i\in\{1,\ldots,n\}$, $\eta^i$ and $\zeta^i$ are related by a sequence of slides and ambient isotopies.
\item
The two $(n+1)$-tuples $\mathcal{D}$ and $\mathcal{D}'$ are called \textit{slide-diffeomorphic} if there exists a diffeomorphism $f:\Sigma\to \Sigma'$ such that $(f(\Sigma);f(\eta^1),\ldots,f(\eta^n))$ and $\mathcal{D}'$ are slide-equivalent.
Here, $f(\eta^i)$ is the set $\{f(\eta^i_1),\ldots,f(\eta^i_n)\}$ of curves mapped by $f$.
\end{itemize}
\end{definition}

\begin{definition}\label{def:rtd}
Let $g$, $k$, $p$, and $b$ be integers such that $g,k,p\geq0$, $b\geq1$, and $2p+b-1\leq k\leq g+p+b-1$.
Let $\Sigma$ be a surface diffeomorphic to $\Sigma_{g,b}$, and let $\alpha=\{\alpha_1,\ldots,\alpha_{g-p}\}$, $\beta=\{\beta_1,\ldots,\beta_{g-p}\}$, and $\gamma=\{\gamma_1,\ldots,\gamma_{g-p}\}$ be sets of $g-p$ pairwise disjoint simple closed curves on $\Sigma$.
The $4$-tuple $\mathcal{D}=(\Sigma;\alpha,\beta,\gamma)$ is called a $(g,k;p,b)$-\textit{relative trisection diagram} if each of the triples $(\Sigma;\alpha,\beta)$, $(\Sigma;\beta,\gamma)$, and $(\Sigma;\gamma,\alpha)$ is slide-diffeomorphic to the standard diagram $(\Sigma_{g,b};\delta,\epsilon)$ shown in Figure \ref{fig:td-gkpb}.
Here, the sets of curves $\delta=\{\delta_1,\ldots,\delta_{g-p}\}$ and $\epsilon=\{\epsilon_1,\ldots,\epsilon_{g-p}\}$ are depicted by the red and blue curves, respectively.
\end{definition}
\begin{figure}[!tbp]
\centering
\includegraphics[scale=1]{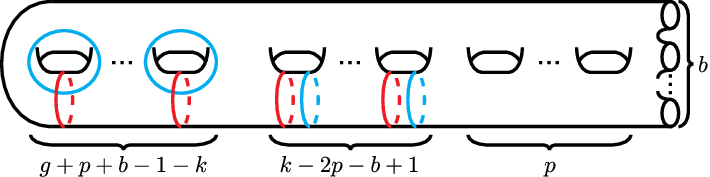}
\caption{The standard diagram $(\Sigma_{g,b}; \delta, \epsilon)$.}
\label{fig:td-gkpb}
\end{figure}

For a relative trisection diagram $\mathcal{D}=(\Sigma;\alpha,\beta,\gamma)$, we denote $\alpha$, $\beta$, and $\gamma$ by red, blue, and green curves, respectively.
See Figure~\ref{fig:td-S2xD2} for an example.
As illustrated on the right side of the figure, we often represent relative trisection diagrams planarly by using the identifications shown in Figure~\ref{fig:diskpair}.
\begin{figure}[!tbp]
\centering
\includegraphics[scale=1]{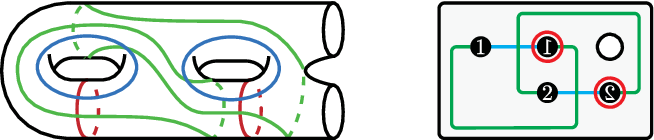}
\caption{Left: a $(2,1;0,2)$-relative trisection diagram. Right: a planar representation of the left diagram.}
\label{fig:td-S2xD2}
\end{figure}
\begin{figure}[!tbp]
\centering
\includegraphics[scale=1]{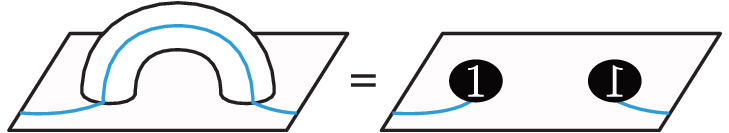}
\caption{A representation of a cylinder by a pair of labeled disks.}
\label{fig:diskpair}
\end{figure}

There is a correspondence between relative trisections and relative trisection diagrams.

\begin{theorem}[{\cite[Theorem~3]{CasGayPin18_1}}]\label{thm:rt-rtd}
The following hold.
\begin{enumerate}
\item
 For any $(g,k;p,b)$-relative trisection diagram $(\Sigma;\alpha,\beta,\gamma)$, there exists a unique (up to diffeomorphism) trisected $4$-manifold $X=X_1\cup X_2\cup X_3$ satisfying the following conditions.
\begin{itemize}
\item $X_1\cap X_2\cap X_3\cong \Sigma$.
\item Under the above identification, each of $\alpha$, $\beta$, and $\gamma$ curves bound compressing disks of $X_3\cap X_1$, $X_1\cap X_2$, $X_2\cap X_3$, respectively.
\end{itemize}
\item
For any relative trisection $\mathcal{T}$, there exists a relative trisection diagram $\mathcal{D}$ such that $\mathcal{T}$ is induced from $\mathcal{D}$ by (i).
\item
Let $\mathcal{D}$ and $\mathcal{D}'$ be relative trisection diagrams.
If the induced relative trisections $\mathcal{T}_{\mathcal{D}}$ and $\mathcal{T}_{\mathcal{D}'}$ are diffeomorphic, then $\mathcal{D}$ and $\mathcal{D}'$ are slide-diffeomorphic.
\end{enumerate}
\end{theorem}

This theorem means the following one-to-one correspondence.
\begin{equation*}
\frac{\{\text{relative trisections}\}}{\text{diffeomorphism}} \ {\stackrel{1:1}{\longleftrightarrow}} \ \frac{\{\text{relative trisection diagrams}\}}{\text{surface diffeomorphisms, slides of curves}}.
\end{equation*}

%
\section{Lower bounds for trisection genus}\label{sect:lower}

In this section, we study lower bounds for the trisection genus of $4$-manifolds with boundary.
For any knot trace $X_m(K)$, the Euler characteristic is equal to $2$ since $X_m(K)$ admits a handle decomposition consisting of a single $0$-handle and a single $2$-handle.
First, we restrict the possible parameters $(g,k;p,b)$ for relative trisections of $4$-manifolds with Euler characteristic $2$.

\begin{lemma}\label{lem:gkpb}
Let $X$ be a $4$-manifold with boundary such that $\chi(X) = 2$.
Suppose that $X$ admits a $(g,k;p,b)$-relative trisection.
Then $g \geq 1$, and the following hold:
\begin{itemize}
\item If $g=1$, then $(1,k;p,b)=(1,0;0,1)$.
\item If $g=2$, then $(2,k;p,b)=(2,1;0,2)$.
\item If $g=3$, then $(3,k;p,b)=(3,2;0,3)$.
\item If $g=4$, then $(4,k;p,b)=(4,3;0,4)$, $(4,1;0,1)$, or $(4,2;1,1)$.
\end{itemize}
\end{lemma}

\begin{proof}
If a $4$-manifold with boundary admits a genus-$0$ relative trisection, then it is diffeomorphic to a $4$-dimensional $1$-handlebody $\natural^k(S^1 \times D^3)$ for some $k \geq 0$.
Its Euler characteristic satisfies $\chi(\natural^k S^1 \times D^3) = 1 - k\leq1$.
Therefore, any $4$-manifold $X$ with $\chi(X) = 2$ cannot admit a genus-$0$ relative trisection.

By the definition of relative trisections and Proposition~\ref{prop:Euler-rel}, if $X$ admits a $(g,k;p,b)$-relative trisection, then the following conditions hold:
\begin{gather}
g,k,p\geq0, \quad b\geq1 \label{eq:rt-def1} \\
2p+b-1 \leq k \leq g+p+b-1 \label{eq:rt-def2} \\
\chi(X) = g-3k+3p+2b-1 \label{eq:rt-Euler1}
\end{gather}

We now consider the case $g=1$.
From the inequalities (\ref{eq:rt-def1}) and (\ref{eq:rt-def2}), the possible types of $(1,k;p,b)$ are listed below:
\begin{equation*}
(1,b-1;0,b), \quad (1,b;0,b), \quad (1,b+1;1,b).
\end{equation*}
Substituting $(g,k;p,b)=(1,b-1;0,b)$ into the Euler characteristic formula~(\ref{eq:rt-Euler1}), we see that $\chi(X)=2$ if and only if $b=1$.
In the other two cases, we obtain $b \leq 0$, which is not admissible.

Next, we consider the case $g=2$. From (\ref{eq:rt-def1}) and (\ref{eq:rt-def2}), the possible types of $(2,k;p,b)$ are listed below:
\begin{align*}
&(2,b-1;0,b), \quad (2,b;0,b), \quad (2,b+1;0,b), \\
&(2,b+1;1,b), \quad (2,b+2;1,b), \quad (2,b+3;2,b).
\end{align*}
By the formula~(\ref{eq:rt-Euler1}), we see that $\chi(X)=2$ holds only when $(2,b-1;0,b)$ with $b=2$.
In all the other cases, we obtain $b\leq0$, which is not admissible.

We consider the case $g=3$. From (\ref{eq:rt-def1}) and (\ref{eq:rt-def2}), the possible types of $(3,k;p,b)$ are listed below:
\begin{align*}
&(3,b-1;0,b), \quad (3,b;0,b), \quad (3,b+1;0,b), \quad (3,b+2;0,b), \\
&(3,b+1;1,b), \quad (3,b+2;1,b), \quad (3,b+3;1,b), \\
&(3,b+3;2,b), \quad (3,b+4;2,b), \quad (3,b+5;3,b).
\end{align*}
By the formula~(\ref{eq:rt-Euler1}), we see that $\chi(X)=2$ holds only when $(g,k;p,b)=(3,b-1;0,b)$ with $b=3$.
In all the other cases, we obtain $b\leq0$, which is not admissible.

Finally, we consider the case $g=4$. From (\ref{eq:rt-def1}) and (\ref{eq:rt-def2}), the possible types of $(4,k;p,b)$ are listed below:
\begin{align*}
&(4,b-1;0,b), \quad (4,b;0,b), \quad (4,b+1;0,b), \quad (4,b+2;0,b), \quad (4,b+3;0,b) \\
&(4,b+1;1,b), \quad (4,b+2;1,b), \quad (4,b+3;1,b), \quad (4,b+4;1,b), \\
&(4,b+3;2,b), \quad (4,b+4;2,b), \quad (4,b+5;2,b), \\
&(4,b+5;3,b), \quad (4,b+6;3,b), \quad (4,b+7;4,b).
\end{align*}
By the formula~(\ref{eq:rt-Euler1}), we see that $\chi(X) = 2$ holds in the cases
$(4,b-1;0,b)$ with $b = 4$,
$(4,b;0,b)$ with $b = 1$, or
$(4,b-1;1,b)$ with $b = 1$.
All other possibilities yield $b \leq 0$, which is not admissible.
\end{proof}

Note that the possible types of $(g,k;p,b)$ for the case $g \geq 5$ can be determined by a similar argument.
As an application of Lemma~\ref{lem:gkpb}, we obtain the following theorem.

\begin{theorem}\label{thm:lb-tg-chi2}
Let $X$ be a $4$-manifold with boundary such that $\chi(X) = 2$.
Then we have $g(X)\geq 1$, and the following hold:
\begin{itemize}
\item
If $\partial{X}$ is not homeomorphic to $S^3$, then $g(X)\geq 2$.
\item
If $\partial{X}$ is not homeomorphic to ${L(p,1)}$ for any $p\in\ZZ$, then $g(X)\geq 3$.
\item
If $\partial{X}$ is not homeomorphic to $M(0;a_1,a_2,a_3)$ for any $a_1,a_2,a_3\in\ZZ$, then $g(X)\geq 4$. See Figure~\ref{fig:Kd-Seif3sing} for the $3$-manifold $M(0;a_1,a_2,a_3)$.
%
\end{itemize}
\end{theorem}

\begin{proof}
By Lemmas~\ref{lem:obd} and \ref{lem:gkpb}, for $g \in \{1,2,3\}$, if $X$ admits a genus-$g$ relative trisection, then $\partial{X}$ admits an open book decomposition with page $\Sigma_{0,g}$.
For $b\leq 3$, let $M$ be a closed $3$-manifold admitting a planar open book decomposition with $b$ binding components.
Then $M$ is classified as follows (see e.g. \cite{Ari08}).
\begin{itemize}
\item
If $b = 1$, then $M$ is homeomorphic to $S^3$.
\item
If $b = 2$, then $M$ is homeomorphic to the lens space $L(p,1)$ for some $p \in \ZZ$.
\item
If $b = 3$, then $M$ is homeomorphic to $M(0;a_1,a_2,a_3)$ for some $a_1,a_2,a_3\in\ZZ$.
\end{itemize}
%
%
\end{proof}

\begin{remark}\label{rem:Mabc}
If at least one of $a_1,a_2,a_3$ is zero, then $M(0;a_1,a_2,a_3)$ is the connected sum of two lens spaces of the form $L(*,1)$.
Conversely, if none of $a_1,a_2,a_3$ is zero, then $M(0;a_1,a_2,a_3)$ is a irreducible manifold, in particular, a Seifert fibered space.
If $a_1,a_2,a_3\in\mathbb{Z}-\{0\}$ and at least one of them is $\pm1$, then $M(0;a_1,a_2,a_3)$ is a lens space.
For details, see Section~4 of~\cite{Ari08}.
%
%
By Theorem~\ref{thm:lb-tg-chi2}, for a $4$-manifold with Euler characteristic $2$, if its boundary $3$-manifold is hyperbolic, then the trisection genus must be at least $4$.
\end{remark}

For similar lower bounds on the trisection genus of $4$-manifolds with arbitrary Euler characteristic, see Section~4 of \cite{Tak25_1}.

\section{Determining the trisection genus of knot traces}\label{sect:determining}

In this section, we study the trisection genus of the knot traces obtained by several specific knots.
We give proofs of Proposition~\ref{thm:tg-unknot} and  Theorems~\ref{thm:tg-Tp,p+1}, \ref{thm:pretzel}, and \ref{thm:tg-fig8}.

\subsection{Unknot}

In this subsection, we prove the following theorem.

\mainunknot*

The $4$-manifold $X_m(U)$ is the disk bundle over $S^2$ with Euler number $m$.
In particular, we have $X_{\pm 1}(U)\cong \pm\CC{P}^2-\Int{D^4}$, which admit $(1,0;0,1)$-relative trisections obtained from the $(1,0)$-trisections of $\pm\CC{P}^2$ by removing a $4$-ball.
The following proposition gives a genus-$2$ relative trisection diagram of $X_m(U)$ for arbitrary $m$.

\begin{proposition}\label{prop:2102rtd-unknot}
Let $\calD_{0} = (\Sigma, \alpha, \beta, \gamma)$ be the left diagram of Figure~\ref{fig:td-unknot}.
For each integer $m$, we define $\calD_{m} = (\Sigma, \alpha, \beta, \tau_c^m(\gamma))$, where $\tau_c^m(\gamma)$ is obtained from $\gamma$ in $\mathcal{D}_{0}$ by applying an $m$-fold Dehn twist $\tau_c^m$ along the curve $c$ (see the right side of Figure~\ref{fig:td-unknot}).
Then $\calD_m$ is a $(2,1;0,2)$-relative trisection diagram of the $m$-trace of the unknot $U$.
\begin{figure}[!htbp]
\centering
\includegraphics[scale=1]{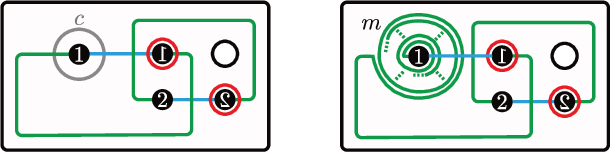}
\caption{Left: $\calD_0$. Right: $\calD_m$.}
\label{fig:td-unknot}
\end{figure}
\end{proposition}

\begin{proof}
We first verify that $\calD_m$ is a $(2,1;0,2)$-relative trisection diagram.
We want to prove that the triples $(\Sigma; \alpha, \beta)$, $(\Sigma; \beta, \gamma)$, and $(\Sigma; \gamma, \alpha)$ are slide-diffeomorphic to the standard diagram shown in Figure~\ref{fig:td-gkpb}.
It is easy to see that $(\Sigma; \alpha, \beta)$ is standard (see Figure~\ref{fig:td-O-Sab}).
\begin{figure}[!htbp]
\centering
\includegraphics[scale=1]{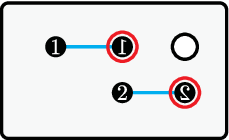}
\caption{$(\Sigma; \alpha, \beta)$.}
\label{fig:td-O-Sab}
\end{figure}

To prove that the remaining two cases, $(\Sigma; \beta, \gamma)$ and $(\Sigma; \gamma, \alpha)$, are also slide-diffeomorphic to the standard diagram, we use the five operations shown in Figure~\ref{fig:opes}.
\begin{figure}[!tbp]
\centering
\includegraphics[scale=1]{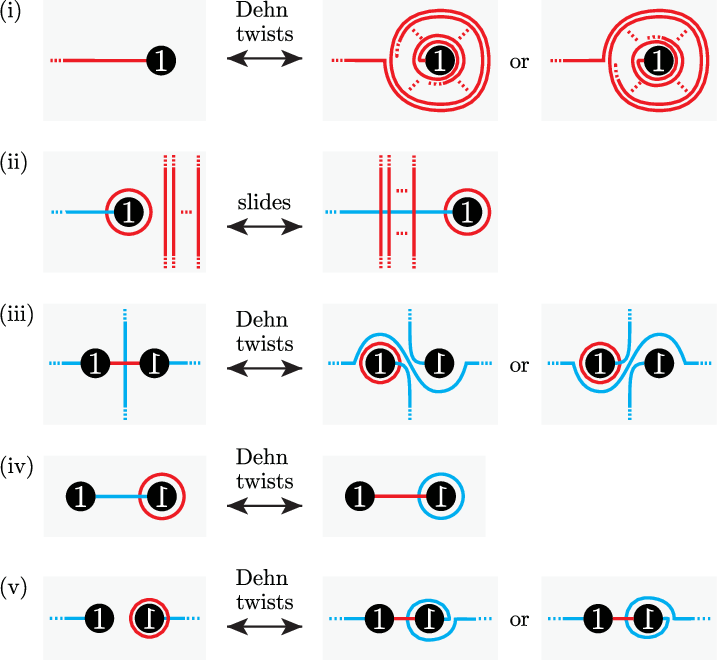}
\caption{The operations (i), (ii), (iii), (iv), and (v).}
\label{fig:opes}
\end{figure}
These operations are obtained by surface diffeomorphisms and slides of curves (see Section~3 of \cite{Tak25_1}).

The case of $(\Sigma; \gamma, \alpha)$ is relatively simple.
We can modify it into the standard diagram by surface diffeomorphisms and isotopies of the curves, as shown in Figure~\ref{fig:td-O-Sca}.
\begin{figure}[!tbp]
\centering
\includegraphics[scale=1]{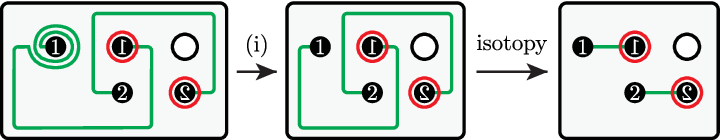}
\caption{Surface diffeomorphisms and slides of curves modifying $(\Sigma; \gamma, \alpha)$ into the standard diagram (for the case $m=2$).}
\label{fig:td-O-Sca}
\end{figure}
We first apply Dehn twists $\tau^{-m}_c$ along the circle $c$, and then obtain the final diagram by ambient isotopies of surfaces.

The triple $(\Sigma; \beta, \gamma)$ can be modified into the standard diagram moves shown in Figure~\ref{fig:td-O-Sbc}.
\begin{figure}[!tbp]
\centering
\includegraphics[scale=1]{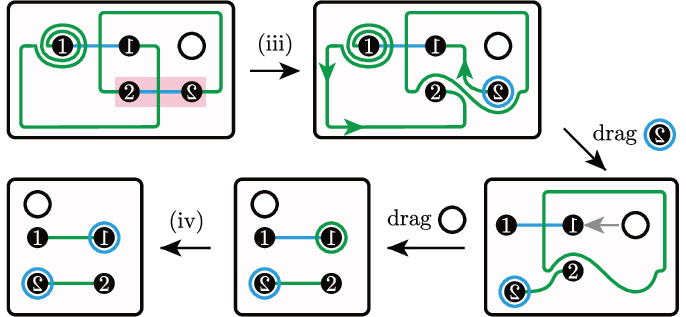}
\caption{Surface diffeomorphisms and slides of curves modifying $(\Sigma;\beta, \gamma)$ into the standard diagram (for the case $m=2$).}
\label{fig:td-O-Sbc}
\end{figure}
First, we apply the operation~(iii) to the highlighted region in the first diagram.
Next, we drag the black disk labeled ``2'', enclosed by the blue circle, along the indicated green curve.
During this process, when the black disk approaches the horizontal blue curve, it can pass through it by applying the operation~(ii).
Note that we can ignore the number of rotations of a green curve around a black disk by applying operation~(i).
After dragging the hole, we apply operation~(iv), which yields the desired diagram.

Next, we prove that the $4$-manifold represented by $\calD_m$ is diffeomorphic to $X_m(U)$.
To do so, we use the algorithm introduced by Kim and Miller~\cite[Section~2.2]{KimMil20}, which gives a method for drawing a handlebody diagram from a relative trisection diagram.
Applying this algorithm to the diagram $\calD_m$ in Figure~\ref{fig:td-unknot} (right), we obtain the handlebody diagram shown in Figure~\ref{fig:kd-unknot} (left).
By performing handle cancellation, we recover the desired diagram.
%
\begin{figure}[!htbp]
\centering
\includegraphics[scale=1]{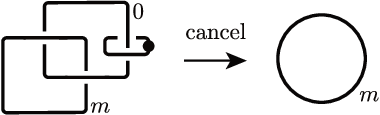}
\caption{A handlemove of the 4-manifold induced by $\calD_m$.}
\label{fig:kd-unknot}
\end{figure}
\end{proof}

For an alternative construction of a $(2,1;0,2)$-relative trisection diagram of $X_m(U)$, see Section~5.1 of~\cite{CasGayPin18_1}.

\begin{proof}[Proof of Proposition~\ref{thm:tg-unknot}]
By Proposition~\ref{prop:2102rtd-unknot}, we see that $g(X_m(U))\leq2$ for any $m\in\ZZ$.
The $\pm1$-trace $X_{\pm 1}(U) \cong \pm\CC{P^2}-\Int D^4$ admit $(1,0;0,1)$-relative trisections, hence $g(X_{\pm1}(U))\leq1$.
The boundary $S^3_{m}(U)$ is homeomorphic to $S^3$ if and only if $m = \pm1$.
For $m\in\ZZ-\{\pm1\}$, the boundary is homeomorphic to the lens space $-L(m,1)$.
Therefore, by Theorem~\ref{thm:lb-tg-chi2}, we have $g(X_m(U)) \geq 1$ if $m = \pm1$, and $g(X_m(U)) \geq 2$ otherwise.
\end{proof}

\subsection{Torus knots}

In this subsection, we determine or give bounds for the trisection genus of traces for several torus knots. First, we prove the following theorem.

\maintorus*

The next proposition gives a genus-$3$ relative trisection diagrams of the traces of $T_{p,p+1}$.

\begin{proposition}\label{prop:3203rtd-torus}
Let $\calD_{0,0} = (\Sigma, \alpha, \beta, \gamma)$ be the left diagram in Figure~\ref{fig:td-Tpp+1}.
For integers $k$ and $l$, we define
\begin{equation*}
\calD_{k,l}=(\Sigma, \alpha, \beta, {\tau_{c_1}^k}\circ{\tau_{c_2}^l}(\gamma)),
\end{equation*}
where ${\tau_{c_1}^k}\circ{\tau_{c_2}^l}(\gamma)$ is obtained from $\gamma$ in $\mathcal{D}_{0,0}$ by applying Dehn twists along the curves $c_1$ and $c_2$ (see the right side of Figure~\ref{fig:td-Tpp+1}).
Then, $\calD_{k,l}$ is a $(3,2;0,3)$-relative trisection diagram of the $(k^2+3k+l+1)$-trace of the torus knot $T_{k+1,k+2}$.
\begin{figure}[!htbp]
\centering
\includegraphics[scale=1]{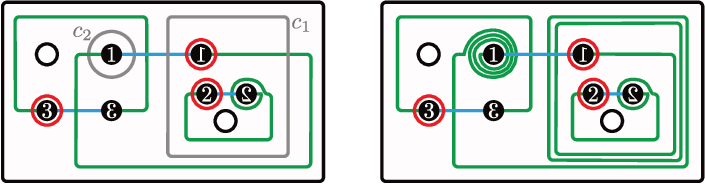}
\caption{Left: $\calD_{0,0}$. Right: $\calD_{k,l}$ for the case $(k,l)=(2,3)$.}
\label{fig:td-Tpp+1}
\end{figure}
\end{proposition}

\begin{proof}
The proof is similar to that of Proposition~\ref{prop:2102rtd-unknot}.
First, we verify that $\calD_{k,l}$ is a $(3,2;0,3)$-relative trisection diagram.
It is easy to see that $(\Sigma; \alpha, \beta)$ and $(\Sigma; \gamma, \alpha)$ can be modified into the standard diagram, so we omit the details.
The sequence of moves modifying $(\Sigma;\gamma,\alpha)$ into the standard diagram is shown in Figure~\ref{fig:td-Sbc-Tpp+1}.
\begin{figure}[!tbp]
\centering
\includegraphics[scale=1]{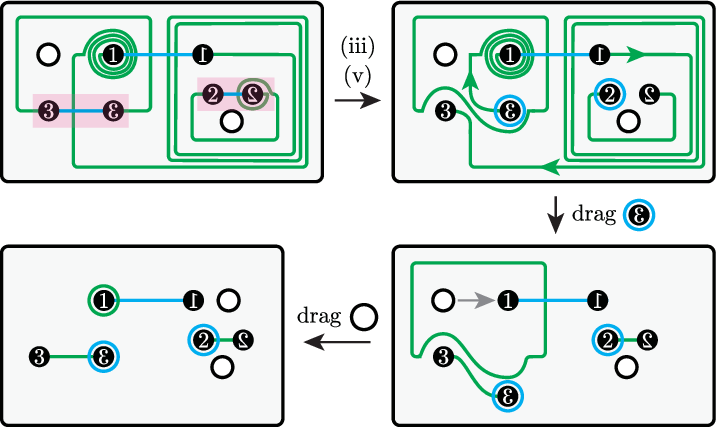}
\caption{Surface diffeomorphisms and slides of curves modifying $(\Sigma;\beta, \gamma)$ into the standard diagram (for the case $k=3$ and $l=2$).}
\label{fig:td-Sbc-Tpp+1}
\end{figure}

Next, we prove that the diagram $\calD_{k,l}$ represents the $(k^2+3k+l+1)$-trace of $T_{k+1,k+2}$.
Applying the algorithm given in \cite[Section~2.2]{KimMil20}, we obtain a handlebody diagram of the $4$-manifold induced by $\calD_{k,l}$, which is shown in the first diagram of Figure~\ref{fig:kd-Tpp+1}.
By the subsequent handle moves shown in the figure, we arrive at the desired diagram.
\begin{figure}[!htbp]
\centering
\includegraphics[scale=1]{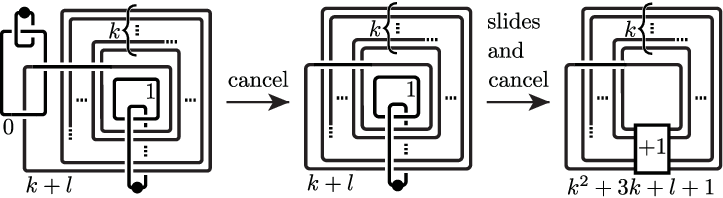}
\caption{Handle moves of $X_{k^2+3k+l+1}(T_{k+1,k+2})$.}
\label{fig:kd-Tpp+1}
\end{figure}
\end{proof}

The next proposition gives a genus-$2$ relative trisection diagrams of the $5$-trace of the trefoil knot.

\begin{proposition}\label{prop:2102rtd-trefoil}
Let $\calD$ be the diagram in Figure~\ref{fig:td-tref}.
Then, $\calD$ is a $(2,1;0,2)$-relative trisection diagram of the $5$-trace of the right-handed trefoil knot $T_{2,3}$.
Thus, the mirror image of $\calD$ is a $(2,1;0,2)$-relative trisection diagram of the $-5$-trace of the left-handed trefoil knot $-T_{2,3}$.
\end{proposition}
\begin{figure}[!htbp]
\centering
\includegraphics[scale=1]{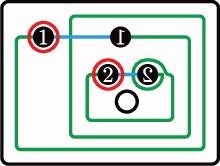}
\caption{$\calD$.}
\label{fig:td-tref}
\end{figure}

\begin{proof}
First, we verify that $\calD$ is a $(2,1;0,2)$-relative trisection diagram.
It is easy to see that $(\Sigma; \alpha, \beta)$ and $(\Sigma; \gamma, \alpha)$ can be modified into the standard diagram.
We can modify $(\Sigma;\gamma,\alpha)$ into the standard diagram using the similar operation as in Figure~\ref{fig:td-Sbc-Tpp+1}.

Next, we prove that the diagram $\calD$ represents the $5$-trace of $T_{2,3}$.
Applying the algorithm given in \cite[Section~2.2]{KimMil20}, we obtain a handlebody diagram of the $4$-manifold induced by $\calD$, which is shown in the first diagram of Figure~\ref{fig:kd-tref}.
By the subsequent handle moves shown in the figure, we arrive at the desired diagram.
\begin{figure}[!htbp]
\centering
\includegraphics[scale=1]{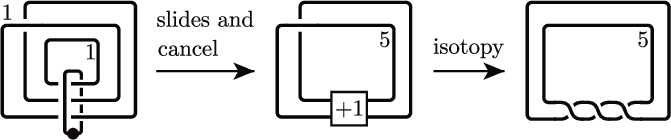}
\caption{Handle moves of $X_5(T_{2,3})$.}
\label{fig:kd-tref}
\end{figure}
\end{proof}

We are now ready to prove the main result of this subsection.

\begin{proof}[Proof of Theorem~\ref{thm:tg-Tp,p+1}]
Proposition~\ref{prop:3203rtd-torus} implies that $g(X_m(\pm T_{p,p+1})) \leq 3$ for any $m\in\ZZ$ and $p\in\ZZ_{\geq2}$.
Indeed, the framing $k^2+3k+l+1$ can take any value by a suitable choice of $l$.
Proposition~\ref{prop:2102rtd-trefoil} gives $g(X_{5}(T_{2,3})) \leq 2$ and $g(X_{-5}(-T_{2,3})) \leq 2$.

To obtain lower bounds of the trisection genus, we apply Theorem~\ref{thm:lb-tg-chi2}.
For any $m\in\ZZ$ and $p\in\ZZ_{\geq2}$, we have $g(X_m(T_{p,p+1})) \geq 2$ since the boundary $S^3_{m}(T_{p,p+1})$ is not homeomorphic to $S^3$.
Furthermore, if $(p,m) \neq (2,5)$, then $g(X_m(T_{p,p+1})) \geq 3$ since $S^3_m(T_{p,p+1})$ is not homeomorphic to any lens space $L(a,1)$. 
Indeed, according to the results of \cite[Theorem~1.1]{KroMroOzsSza07} and \cite[Theorem~9]{Tan09}, if $S^3_m(K) \cong L(a,1)$ for a knot $K$ and an integer $m$, then the pair $(K,m)$ must be one of the following:
\begin{equation*}
(U,m) \text{ with } m \in \ZZ, \quad (T_{2,3},5), \quad \text{or } (-T_{2,3},-5).
\end{equation*}
\end{proof}


We also prove the following theorem.

\maintoruspq*

\begin{proof}
Let $\calD_{0,0,0} = (\Sigma, \alpha, \beta, \gamma)$ be the left diagram in Figure~\ref{fig:td-Tpq}.
For integers $j$, $k$, and $l$, we define
\begin{equation*}
\calD_{j,k,l}=(\Sigma; \alpha, \beta, {\tau_{c_1}^j}\circ{\tau_{c_2}^k}\circ{\tau_{c_3}^l}(\gamma)),
\end{equation*}
where ${\tau_{c_1}^j}\circ{\tau_{c_2}^k}\circ{\tau_{c_3}^l}(\gamma)$ is obtained from $\gamma$ in $\mathcal{D}_{0,0,0}$ by applying Dehn twists along the curves $c_1$, $c_2$, and $c_3$ (see the right side of Figure~\ref{fig:td-Tpq}).
We prove that $\calD_{j,k,l}$ is a $(4,3;0,4)$-relative trisection diagram of the $(j-2)k^2+(2j-3)k+(j-2+l)$-trace of the $(k+1, (j-2)(k+1)+1)$-torus knot.
\begin{figure}[!htbp]
\centering
\includegraphics[scale=1]{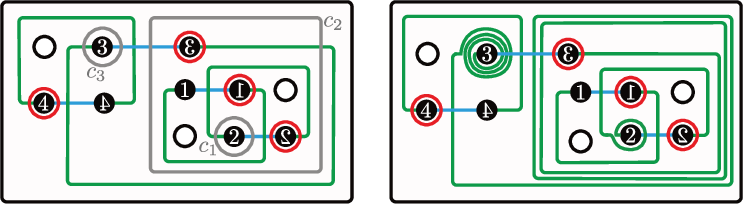}
\caption{Left: $\calD_{0,0,0}$. Right: $\calD_{j,k,l}$ for the case $(j,k,l)=(1,2,3)$.}
\label{fig:td-Tpq}
\end{figure}

We verify that $\calD_{j,k,l}$ is a $(4,3;0,4)$-relative trisection diagram.
It is easy to see that $(\Sigma; \alpha, \beta)$ and $(\Sigma; \gamma, \alpha)$ can be modified into the standard diagram, so we omit the details.
Regarding $(\Sigma; \beta, \gamma)$, it can also be modified into the standard diagram by combining the operations shown in Figures~\ref{fig:td-O-Sbc} and \ref{fig:td-Sbc-Tpp+1}.

Next, we prove that the diagram $\calD_{j,k,l}$ represents the $(j-2)k^2+(2j-3)k+(j-2+l)$-trace of the $(k+1, (j-2)(k+1)+1)$-torus knot.
Applying the algorithm given in \cite[Section~2.2]{KimMil20}, we obtain a handlebody diagram of the $4$-manifold induced by $\calD_{j,k,l}$, which is shown in the first diagram of Figure~\ref{fig:kd-Tpq}.
By the subsequent handle moves shown in the figure, we arrive at the desired diagram.
Note that the framing $(j-2)k^2+(2j-3)k+(j-2+l)$ can take any value by a suitable choice of $l$.
\begin{figure}[!tbp]
\centering
\includegraphics[scale=1]{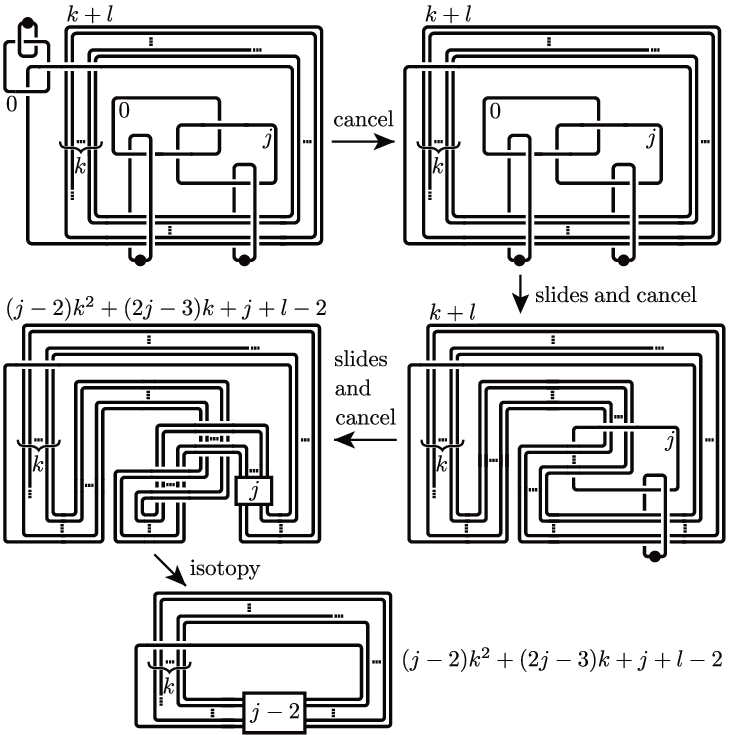}
\caption{Handle moves of the $(j-2)k^2+(2j-3)k+(j-2+l)$-trace of the $(k+1, (j-2)(k+1)+1)$-torus knot.}
\label{fig:kd-Tpq}
\end{figure}
\end{proof}

\subsection{$(-2, 3, 2n+1)$-pretzel knots}

We study the trisection genus of the traces of the $(-2,3,2n+1)$-pretzel knots shown in Figure~\ref{fig:K_n}.
\begin{figure}[!htbp]
\centering
\includegraphics[scale=1]{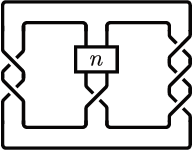}
\caption{$P(-2,3,2n+1)$}
\label{fig:K_n}
\end{figure}

\mainpretzel*

First, we give genus-$4$ relative trisection diagrams that realize the knot traces of $P(-2, 3, 2n+1)$.

\begin{proposition}\label{prop:4304rtd-pretzel}
Let $\calD_{0,0} = (\Sigma, \alpha, \beta, \gamma)$ be the diagram shown in Figure~\ref{fig:td-K_n}.
For integers $k$ and $l$, we define
\begin{equation*}
\calD_{k,l}=(\Sigma, \alpha, \beta, {\tau_{c_1}^k}\circ{\tau_{c_2}^l}(\gamma)),
\end{equation*}
where ${\tau_{c_1}^k}\circ{\tau_{c_2}^l}(\gamma)$ is obtained from $\gamma$ in $\mathcal{D}_{0,0}$ by applying Dehn twists along the curves $c_1$ and $c_2$ (see Figure~\ref{fig:td-K_n-lk}).
Then $\calD_{k,l}$ is a $(4,3;0,4)$-relative trisection diagram of the $(k+4l+4)$-trace of the knot $P(-2,3,2l-1)$.
\begin{figure}[!tbp]
\centering
\includegraphics[scale=1]{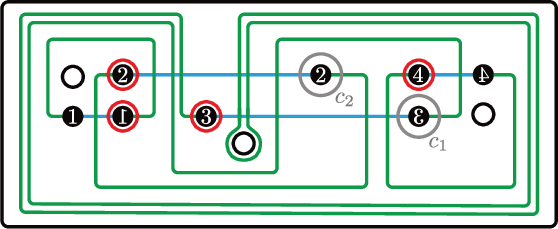}
\caption{$\calD_{0,0}$.}
\label{fig:td-K_n}
\end{figure}
\begin{figure}[!tbp]
\centering
\includegraphics[scale=1]{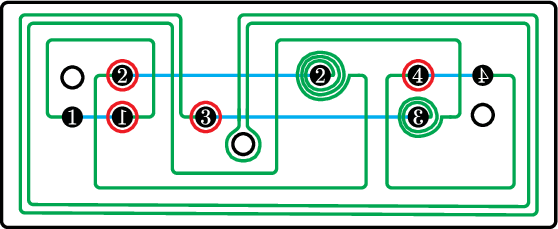}
\caption{$\calD_{k,l}$ for the case $(k,l)=(2,3)$.}
\label{fig:td-K_n-lk}
\end{figure}
\end{proposition}

\begin{proof}
First, we verify that $\calD_{k,l}$ is a $(4,3;0,4)$-relative trisection diagram.
It is easy to see that $(\Sigma; \alpha, \beta)$ and $(\Sigma; \gamma, \alpha)$ can be modified into the standard diagram, so we omit the details.
The sequence of moves modifying $(\Sigma;\gamma,\alpha)$ into the standard diagram is shown in Figure~\ref{fig:td-K_n-Sbc}.
\begin{figure}[!tbp]
\centering
\includegraphics[scale=1]{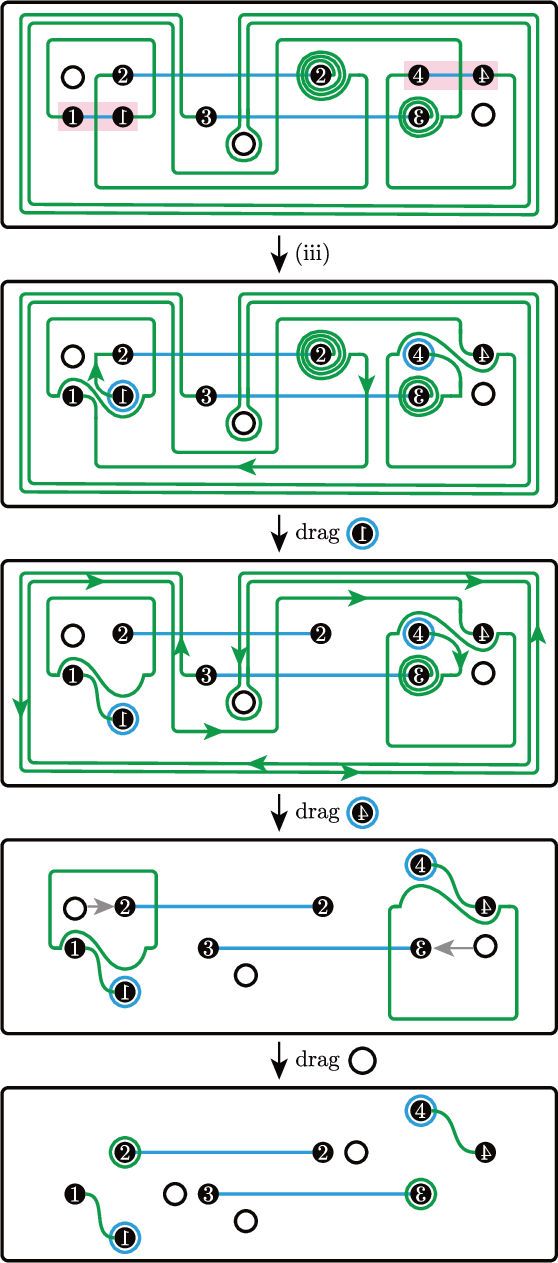}
\caption{Surface diffeomorphisms and slides of curves modifying $(\Sigma;\beta, \gamma)$ into the standard diagram (for the case $k=2$ and $l=3$).}
\label{fig:td-K_n-Sbc}
\end{figure}

Next, we prove that the diagram $\calD_{k,l}$ represents the $(k+4l+4)$-trace of $P(-2,3,2l-1)$.
Applying the algorithm given in \cite[Section~2.2]{KimMil20}, we obtain a handlebody diagram of the $4$-manifold induced by $\calD_{k,l}$, which is shown in the first diagram of Figure~\ref{fig:kd-K_n}.
By the subsequent handle moves shown in the figure, we arrive at the desired diagram.
\end{proof}
\begin{figure}[!tbp]
\centering
\includegraphics[scale=1]{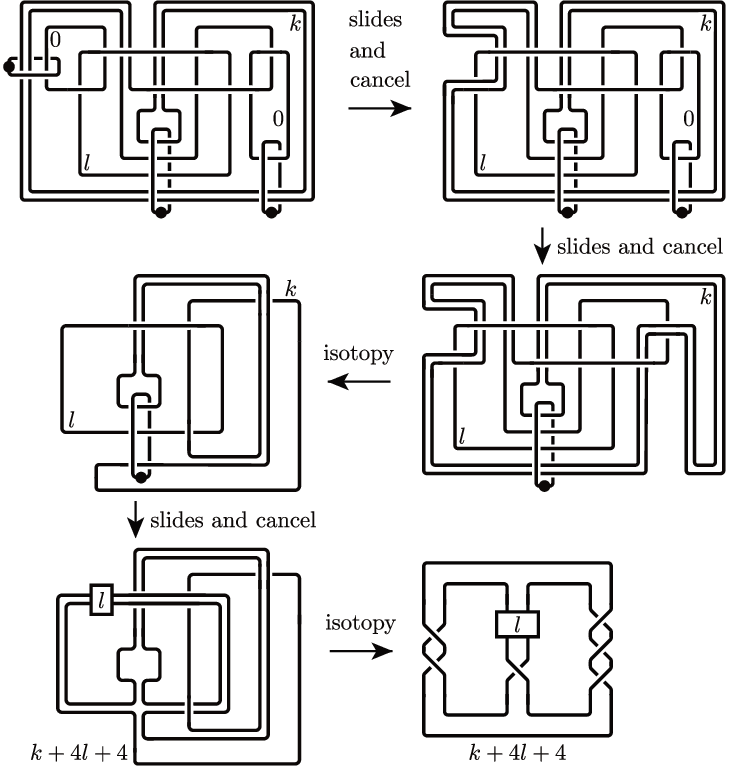}
\caption{Handle moves of the $(k+4l+4)$-trace of $P(-2,3,2l-1)$.}
\label{fig:kd-K_n}
\end{figure}


\begin{proof}[Proof of Theorem~\ref{thm:pretzel}]
Proposition~\ref{prop:4304rtd-pretzel} gives $g(X_m(P(-2,3,2n+1)))\leq4$ for any $m,n\in\ZZ$.
Note that the framing $k+4l+4$ can take any value by a suitable choice of $k$.

Suppose that $n\notin\{0,1,2\}$, so that the pretzel knot $P(-2,3,2n+1)$ is hyperbolic.
We prove that $g(X_m(P(-2,3,2n+1)))\geq4$ for any $m\in\mathbb{Z}$.
By Theorem~\ref{thm:lb-tg-chi2}, it suffices to show that $S^3_m(P(-2,3,2n+1))\not\cong M(0;a_1,a_2,a_3)$ for any $a_1,a_2,a_3\in\mathbb{Z}$.
Recall that $M(0;a_1,a_2,a_3)$ is either reducible or Seifert fibered.
By Wu~\cite[Corollary~2.6]{Wu96}, Dehn surgery on a hyperbolic arborescent knot, including any hyperbolic pretzel knot, does not yield a reducible $3$-manifold.
On the other hand, Meier~\cite[Theorem~1.1]{Mei14} showed that, for $n\notin{0,1,2}$, the surgered manifold $S^3_m(P(-2,3,2n+1))$ is Seifert fibered if and only if $m\in{4n+6,4n+7}$.
In these two cases, we have
\begin{align*}
S^3_{4n+6}(P(-2, 3, 2n+1)) \cong M\left(0; 2, -4, \frac{2n - 5}{2}\right), \\
S^3_{4n+7}(P(-2, 3, 2n+1)) \cong M\left(0; \frac{3}{2}, \frac{-5}{2}, n-2\right).
\end{align*}
It follows from Lemma~\ref{lem:sfs-kubetu}, proved below, that neither of these Seifert fibered spaces can be realized as $M(0;a_1,a_2,a_3)$.
%
%
\end{proof}

\begin{lemma}\label{lem:sfs-kubetu}
Let $a_1, a_2, a_3$ and $n$ be integers with $n\notin\{0,1,2\}$.
Then the Seifert fibered space $M(0;a_1,a_2,a_3)$ is not orientation-preservingly homeomorphic to either
\[
M\left(0;\frac{3}{2},-\frac{5}{2},n-2\right)
\quad\text{or}\quad
M\left(0;2,-4,\frac{2n-5}{2}\right).
\]
\end{lemma}

To prove Lemma~\ref{lem:sfs-kubetu}, we recall the following classification theorem for Seifert fibered spaces (see e.g. \cite[Section~5.3]{Orl72b}, \cite[Section~3]{IssMcC23}).

\begin{theorem}\label{thm-bunrui-sfs}
Let $e, a_i, b_i$ be integers satisfying $\gcd(a_i,b_i)=1$ for $i\in\{1,2,3\}$.
There exists an orientation-preserving homeomorphism 
\[
M\left(e; \frac{a_1}{b_1}, \frac{a_2}{b_2}, \frac{a_3}{b_3}\right) \cong M\left(e'; \frac{a'_1}{b'_1}, \frac{a'_2}{b'_2}, \frac{a'_3}{b'_3}\right)
\]
if and only if the following two conditions hold:
\begin{itemize}
\item the generalized Euler invariants coincide, namely,
\[ e - \frac{b_1}{a_1} - \frac{b_2}{a_2} - \frac{b_3}{a_3} = e' - \frac{b'_1}{a'_1} - \frac{b'_2}{a'_2} - \frac{b'_3}{a'_3} \]
\item there exist integers $k_1, k_2, k_3$ and a bijection $\pi:\{1, 2, 3\}\to\{1,2,3\}$ such that
\[ \frac{b'_i}{a'_i} = \frac{b_{\pi(i)}}{a_{\pi(i)}} + k_i \quad \text{for } i \in \{1, 2, 3\}. \]
\end{itemize}
\end{theorem}

\begin{proof}[Proof of Lemma~\ref{lem:sfs-kubetu}]%
First, assume that $M(0; a_1, a_2, a_3)$ is homeomorphic to $M\left(0; \frac{3}{2}, -\frac{5}{2}, n-2\right)$.
By Theorem~\ref{thm-bunrui-sfs}, there exist integers $k_1, k_2, k_3$ such that
\begin{equation*}
\frac{1}{a_1} = \frac{2}{3} + k_1, \quad
\frac{1}{a_2} = \frac{-2}{5} + k_2, \quad
\frac{1}{a_3} = \frac{1}{n-2} + k_3.
\end{equation*}
Solving the second equality for $a_2$, we obtain $a_2 = 5/(-2 + 5k_2)$.
However, there exists no integer $k_2$ making $a_2$ an integer, which is a contradiction.

Next, assume that $M(0; a_1, a_2, a_3)$ is homeomorphic to $M\left(0; 2, -4, \frac{2n-5}{2}\right)$.
By Theorem~\ref{thm-bunrui-sfs}, there exist integers $k_1, k_2, k_3$ such that
\begin{equation*}
\frac{1}{a_1} = \frac{1}{2} + k_1, \quad
\frac{1}{a_2} = \frac{-1}{4} + k_2, \quad
\frac{1}{a_3} = \frac{2}{2n-5} + k_3.
\end{equation*}
Solving these for $a_1, a_2, a_3$, we obtain
\begin{align*}
a_1 = \frac{2}{1 + 2k_1}, \quad
a_2 = \frac{-4}{1 - 4k_2}, \quad
a_3 = \frac{2n-5}{2 + k_3(2n-5)}.
\end{align*}
Since $a_1, a_2, a_3 \in \mathbb{Z}$, it follows that $k_1 \in \{-1, 0\}$, $k_2 = 0$, and $k_3 \in \{-1, -3\}$.

On the other hand, the generalized Euler invariants of both Seifert fibered spaces must coincide:
\begin{equation*}
0 - \frac{1}{2} - \frac{1}{-4} - \frac{2}{2n-5} =
0 - \frac{1}{a_1} - \frac{1}{a_2} - \frac{1}{a_3}.
\end{equation*}
Substituting ${1}/{a_i}$ into this equation, we obtain $k_1 + k_2 + k_3 = 0$.
However, this contradicts the possible values of $k_1, k_2, k_3$ derived above.
\end{proof}

\begin{remark}
When $n\in\{0,1,2\}$, the pretzel knot $P(-2,3,2n+1)$ is a torus knot.
In particular, $P(-2,3,1)=T_{2,5}$, $P(-2,3,3)=T_{3,4}$, and $P(-2,3,5)=T_{3,5}$ (see e.g. \cite{Oer84}, \cite{BonSie10}).
%
By Theorem~\ref{thm:g=4}, when $n=1$, the trisection genus of $X_m(P(-2,3,3))$ is $3$ for every integer $m$.
It remains unclear whether the trisection genus of the $m$-trace is $4$ for most integers $m$.
\end{remark}


\subsection{Figure-eight knot}

The figure-eight knot $4_1$ also admits traces of trisection genus $4$.
We prove the next theorem.

\mainfigureeight*

\begin{proposition}\label{prop:rtd-4_1}
Let $\calD_0 = (\Sigma, \alpha, \beta, \gamma)$ be the diagram shown in Figure~\ref{fig:td-4_1-0}.
For each integer $m$, we define $\calD_m = (\Sigma, \alpha, \beta, \tau_c^m(\gamma))$,
where $\tau_c^m(\gamma)$ is obtained from $\gamma$ in $\mathcal{D}_0$ by applying an $m$-fold Dehn twist $\tau_c^m$ along the curve $c$ (see Figure~\ref{fig:td-4_1-m}).
Then $\calD_m$ is a $(4,3;0,4)$-relative trisection diagram of the $(m-3)$-trace of the figure-eight knot $4_1$.
\end{proposition}
\begin{figure}[!tbp]
\centering
\includegraphics[scale=1]{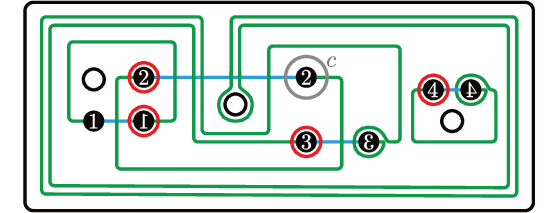}
\caption{$\calD_0$.}
\label{fig:td-4_1-0}
\end{figure}
\begin{figure}[!tbp]
\centering
\includegraphics[scale=1]{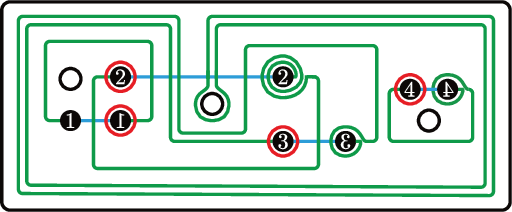}
\caption{$\calD_m$ for the case $m=2$.}
\label{fig:td-4_1-m}
\end{figure}

\begin{proof}
First, we verify that $\calD_{m}$ is a $(4,3;0,4)$-relative trisection diagram.
It is easy to see that $(\Sigma; \alpha, \beta)$ and $(\Sigma; \gamma, \alpha)$ can be modified into the standard diagram, so we omit the details.
The sequence of moves modifying $(\Sigma;\gamma,\alpha)$ into the standard diagram is shown in Figure~\ref{fig:td-4_1-Sbc}.
\begin{figure}[!tbp]
\centering
\includegraphics[scale=1]{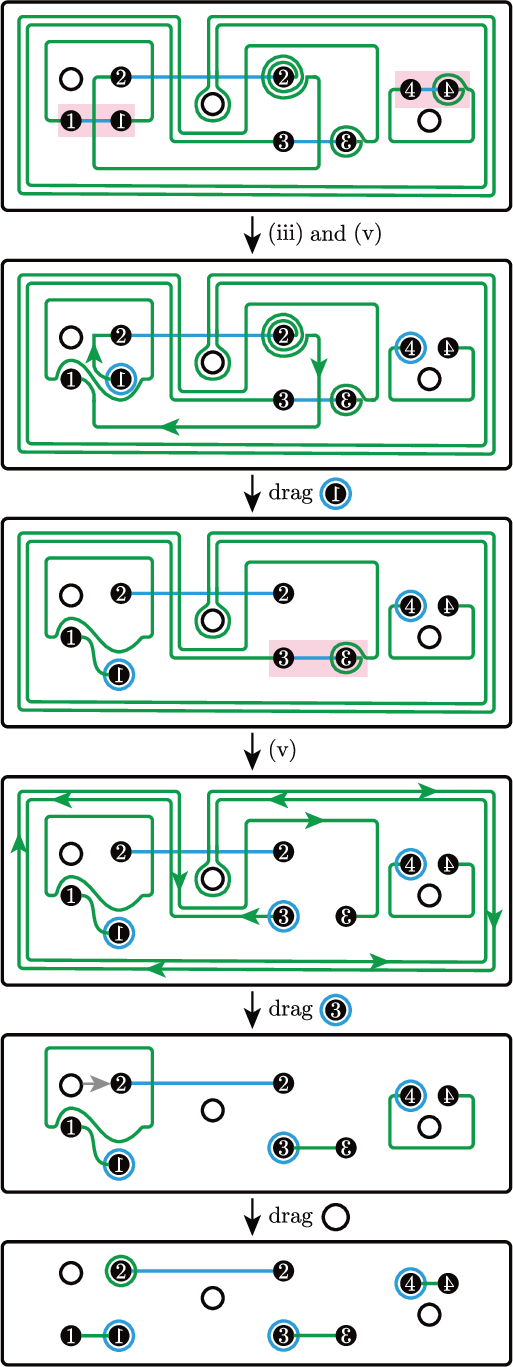}
\caption{Surface diffeomorphisms and slides of curves modifying $(\Sigma;\beta, \gamma)$ into the standard diagram (for the case $m=2$).}
\label{fig:td-4_1-Sbc}
\end{figure}

Next, we prove that the diagram $\calD_{m}$ represents the $(m-3)$-trace of the figure-eight knot.
Applying the algorithm given in \cite[Section~2.2]{KimMil20}, we obtain a handlebody diagram of the $4$-manifold induced by $\calD_{m}$, which is shown in the first diagram of Figure~\ref{fig:kd-4_1}.
By the subsequent handle moves shown in the figure, we arrive at the desired diagram.
\begin{figure}[!tbp]
\centering
\includegraphics[scale=1]{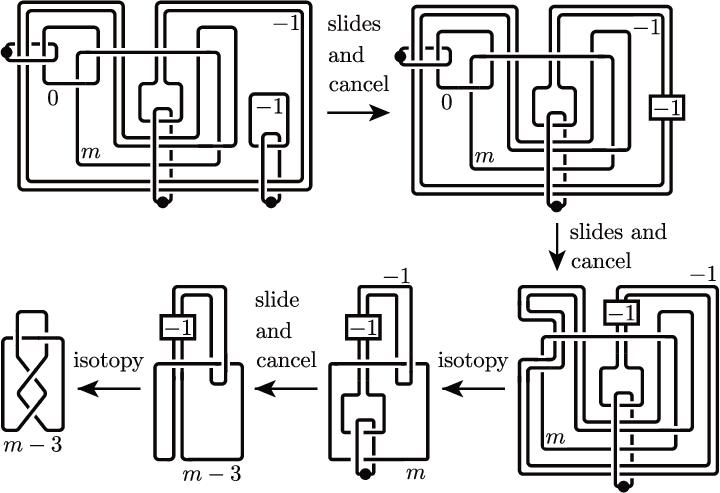}
\caption{Handle moves of the $(m-3)$-trace of the figure-eight knot.}
\label{fig:kd-4_1}
\end{figure}
\end{proof}

\begin{proof}[Proof of Theorem~\ref{thm:tg-fig8}]
By Proposition~\ref{prop:rtd-4_1}, we have $g(X_m(4_1)) \leq 4$ for any integer $m$.
We now give lower bounds for $g(X_m(4_1))$.
The boundary $S^3_{m}(4_1)$ is a Seifert fibered space (not a lens space) if $m \in \{ \pm1, \pm2, \pm3 \}$, a toroidal manifold if $m \in \{0, \pm4\}$, and a hyperbolic manifold otherwise (see e.g.~\cite[Theorem~4.7]{Thu22}).
Note that $S^3_3(4_1)$ is the Seifert fibered space $M(1;3,3,4)$; however, by Theorem~\ref{thm-bunrui-sfs}, we have $M(1;3,3,4) \not\cong M(0;a_1,a_2,a_3)$ for any $a_1,a_2,a_3\in\ZZ$.
Therefore, Theorem~\ref{thm:lb-tg-chi2} implies that
$g(X_m(4_1)) \geq 3$ for any integer $m$, and
$g(X_m(4_1)) \geq 4$ when $m \notin \{\pm1,\pm2\}$.
\end{proof}

\begin{remark}\label{rem:4_1Seifert}
When $m \in \{\pm 1, \pm 2\}$, the boundary $S^3_{m}(4_1)$ is homeomorphic to a Seifert fibered space of the form $(M;0,a_1,a_2,a_3)$. Specifically,
\begin{align*}
S^3_{\pm 1}(4_1) \cong \pm M(1;2,3,7) \cong \pm M(0;-2,3,7), \\
S^3_{\pm 2}(4_1) \cong \pm M(1;2,4,5) \cong \pm M(0;-2,4,5).
\end{align*}
By Lemma~\ref{lem:gkpb}, in these cases, $X_m(4_1)$ could potentially admit a $(3,2;0,3)$-relative trisection.
However, we have not been able to show the existence or non-existence of such a trisection.
\end{remark}

\section{Proof of main results}\label{sect:proofmain}

In this section, we give the proofs of Theorems~\ref{thm:g=12}, \ref{thm:g=3}, \ref{thm:g=4}, and \ref{thm:g>N}.

\maingonetwo*

\begin{proof} 
We have $\chi(X_m(K))=2$ for any knot $K$ and $m\in\ZZ$, since $X_m(K)$ admits a handle decomposition consisting of a single $0$-handle and a single $2$-handle.
By Theorem~\ref{thm:lb-tg-chi2}, $g(X_m(K)) \geq 1$.

Suppose that $g(X_m(K)) = 1$.
By Lemma~\ref{lem:gkpb}, $X_m(K)$ admits a $(1;0;0,1)$-relative trisection.
Applying Lemma~\ref{lem:obd}, we see that the boundary $S^3_m(K)$ admits an open book decomposition with page $\Sigma_{0,1}$, and hence it is homeomorphic to $S^3$.
By the Property~P conjecture (\cite{GorLue89}), $K$ is the unknot and $m = \pm1$.
Conversely, by Proposition~\ref{thm:tg-unknot}, we obtain $g(X_{\pm1}(U)) = 1$.

Suppose that $g(X_m(K)) = 2$.
By Lemma~\ref{lem:gkpb}, $X_m(K)$ admits a $(2;1;0,2)$-relative trisection.
Applying Lemma~\ref{lem:obd}, we see that the boundary $S^3_m(K)$ admits an open book decomposition with page $\Sigma_{0,2}$, and hence it is homeomorphic to $\pm L(m,1)$.
By the results of Kronheimer, Mrowka, Ozsváth, and Szabó~\cite[Theorem~1.1]{KroMroOzsSza07} and Tange~\cite[Theorem~9]{Tan09}, the pair $(K,m)$ must be one of the following:
%
\begin{equation*}
(U,m) \text{ with } m\in\ZZ, \quad (T_{2,3},5), \quad \text{or } (-T_{2,3},-5).
\end{equation*}
Conversely, by Proposition~\ref{thm:tg-unknot} and Theorem~\ref{thm:tg-Tp,p+1}, we obtain $g(X_{m}(U)) = g(X_{5}(T_{2,3})) = g(X_{-5}(-T_{2,3})) = 2$, where $m\neq\pm1$.
\end{proof}

\maingthree*

\begin{proof} 
The first statement follows from Theorem~\ref{thm:tg-Tp,p+1}.
The second one is a consequence of Theorem~\ref{thm:lb-tg-chi2}.
\end{proof}

\maingfour*

\begin{proof} 
The first statement follows from Theorem~\ref{thm:pretzel}.
Note that the family $\{P(-2, 3, 2n+1)\}_{n \in \ZZ}$ consists of mutually distinct knots (see e.g. \cite[Theorem~2.3.1]{Kaw96}).
%

Let $K$ be a hyperbolic knot.
According to Thurston's hyperbolic Dehn surgery theorem~\cite[Theorem 5.8.2]{Thu22}, for all but finitely many $m\in\ZZ$, the boundary $\partial{X_m(K)}$ is hyperbolic and therefore not homeomorphic to $M(a_1, a_2, a_3)$ for any $a_1, a_2, a_3 \in \ZZ$.
By Theorem~\ref{thm:lb-tg-chi2}, we have $g(X_m(K)) \geq 4$ for all such $m$.
\end{proof}

Finally, we prove the following theorem.

\mainlargen*

\begin{proof}
Take an arbitrary positive integer $N$.
Let $K$ be the pretzel knot $P(p_1, p_2, \ldots, p_r)$ with $r$ strands, where the integers $r, p_1, p_2, \ldots, p_r$ satisfy the following conditions:
\begin{itemize}
\item $r$ is an odd integer and $r \geq N$,
\item for any $i \in \{1, 2, \ldots, r\}$, $p_i$ is an odd integer and $p_i = p_{r-i+1}$,
\item $\gcd(p_1, p_2, \ldots, p_r) \neq 1$.
\end{itemize}
For such a knot, it follows from Lemma~2 in \cite{BoiLusMor94} that if $m$ is even, the Heegaard genus of $S^3_{m}(K)$ is equal to $r$.
For example, one may take the pretzel knot $P(3,3,\ldots,3)$ with an odd number of strands at least $N$ (see Figure~\ref{fig:pretzel333}), which satisfies all of the above conditions. 
Then, for every even integer $m$, the Heegaard genus of $S^3_m(K)$ is at least $N$.
\begin{figure}[!htbp]
\centering
\includegraphics[scale=1]{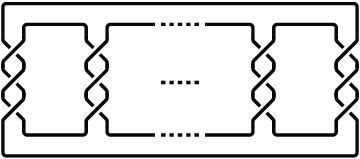}
\caption{The pretzel knot $P(3,3,\ldots,3)$ with an odd number of strands at least $N$.}
\label{fig:pretzel333}
\end{figure}

Let $m$ be any even integer.
Suppose for contradiction that $g(X_m(K)) \leq N$.
Then there exists a $(N,k;p,b)$-relative trisection $\calT$ of $X_m(K)$.
By the definition of relative trisections and Proposition~\ref{prop:Euler-rel}, we have
\begin{gather}
k,p \geq 0, \quad b \geq 1, \notag\\
2p+b-1 \leq k \leq N+p+b-1, \label{eq:rt-def22} \\
N-3k+3p+2b-1 = 2. \label{eq:rt-Euler22}
\end{gather}

By Lemma~\ref{lem:obd}, $\calT$ induces an open book decomposition of the boundary $S^3_{m}(K)$ with page $\Sigma_{p,b}$. Thus, $S^3_{m}(K)$ also admits a Heegaard splitting of genus $2p+b-1$. 

We derive a contradiction by showing that $2p+b-1 \leq N-1$.
Solving the equation~\eqref{eq:rt-Euler22} for $k$, we obtain
\begin{equation*}
k = p-1 + \frac{2b+N}{3}.
\end{equation*}
Substituting this into the inequality $2p+b-1 \leq k$ from \eqref{eq:rt-def22}, we obtain $2p+b-1 \leq p-1+(2b+N)/{3}$,
which simplifies to $3p+b \leq N$.
%
Hence, $2p+b-1 \leq 3p+b-1 \leq N-1$.
This contradicts the fact that the Heegaard genus of $S^3_m(K)$ is at least $N$ for every even integer $m$.
\end{proof}

\section*{Acknowledgements}
The author would like to thank Motoo Tange and Kimihiko Motegi for helpful conversations and comments about exceptional surgeries on pretzel knots.
The author was partially supported by JSPS KAKENHI Grant Numbers 24KJ1561 and 26K24497.





\end{document}